\documentclass{amsart}
\usepackage{amssymb,graphicx,color,cite}
\newtheorem{theorem}{Theorem}[section]

\newtheorem{claim}[theorem]{Claim}

\newtheorem{proposition}[theorem]{Proposition}
\newtheorem{observation}[theorem]{Observation}

\theoremstyle{remark}

\newcommand{\cgA}{\mathcal{A}}

\newcommand{\cgC}{\mathcal{C}}
\newcommand{\cgE}{\mathcal{E}}

\newcommand{\cgR}{\mathcal{R}}

\newcommand{\bbP}{\mathbb{P}}

\newcommand{\Inc}{\operatorname{Inc}}
\newcommand{\se}{\operatorname{se}}
\newcommand{\maxdd}{\operatorname{maxdd}}
\newcommand{\maxud}{\operatorname{maxud}}
\newcommand{\dd}{\operatorname{dd}}
\newcommand{\ud}{\operatorname{ud}}
\newcommand{\uc}{\operatorname{uc}}
\newcommand{\dc}{\operatorname{dc}}
\newcommand{\cdim}{\operatorname{cdim}}
\newcommand{\bdim}{\operatorname{bdim}}
\newcommand{\bfN}{\mathbb{N}}

\newcommand{\bftwo}{\mathbf{2}}
\newcommand{\crit}{\operatorname{crit}}
\newcommand{\fdim}{\operatorname{fdim}}
\newcommand{\ldim}{\operatorname{ldim}}
\newcommand{\vcdim}{\operatorname{vcdim}}
\newcommand{\VC}{\operatorname{VC}}

\newcommand{\ple}{\operatorname{ple}}
\newcommand{\width}{\operatorname{width}}

\begin{document}

\title[CONVEX GEOMETRIES]{Concepts of Dimension for Convex Geometries}

\author[KNAUER]{Kolja Knauer}
\address{Universitat de Barcelona\\
Departament de Mathem\'{a}tiques i Inform\'{a}tica\\
Gran Via de les Corts Catalanes, 585\\
08007 Barcelona\\
Spain}
\email{kolja.knauer@gmail.com}

\author[TROTTER]{William T. Trotter}

\address{School of Mathematics\\
Georgia Institute of Technology\\
Atlanta, Georgia 30332\\
U.S.A.}

\email{wtt.math@gmail.com}

\date{March 15, 2023}

\subjclass[2010]{06A07, 05C35}

\keywords{Convex geometry, partially ordered set, dimension, standard example}

\begin{abstract} 
  Let $X$ be a finite set. A family $P$ of subsets of $X$ is  
  called a \textit{convex geometry} with ground set $X$ if 
  (1)~$\emptyset, X\in P$; (2)~$A\cap B\in P$ whenever 
  $A,B\in P$; and (3)~if $A\in P$ and $A\neq X$, there is an element $\alpha\in X-A$ 
  such that $A\cup\{\alpha\}\in P$.  As a non-empty family of sets, a convex
  geometry has a well defined $\VC$-dimension.  In the literature, a second
  parameter, called convex dimension, has been defined expressly for these
  structures.  Partially ordered by inclusion, a convex geometry is 
  also a poset, and four additional dimension parameters have been defined
  for this larger class, called Dushnik-Miller dimension, 
  Boolean dimension, local dimension, and fractional dimension, respectively.  
  For each pair of these six dimension parameters, we investigate whether there is 
  an infinite class of convex geometries on which one parameter is bounded and
  the other is not.
\end{abstract}

\maketitle

\section{Statement of Results}\label{sec:results}

The primary goal of this paper is to investigate concepts of dimension
for a special class of posets called convex geometries, that are defined as inclusion orders of certain set systems on a finite ground set.  The concepts of dimension that we consider 
are called convex dimension, $\VC$-dimension, Dushnik-Miller dimension,
Boolean dimension, fractional dimension, and local dimension,
denoted $\cdim(P)$, $\vcdim(P)$, $\dim(P)$, $\bdim(P)$, $\ldim(P)$, and $\fdim(P)$,
respectively.  We will also study a related poset parameter, called standard
example number, denoted $\se(P)$.  When $P$ is a convex geometry, the following 
inequalities are known or easily seen to hold: 
\begin{enumerate}
 \item $\cdim(P)\ge\dim(P)\ge\max\{\vcdim(P),\bdim(P),
\fdim(P), \ldim(P),\se(P)\}$, see Proposition~\ref{pro:cdim-dim},
\item $\se(P)\ge\vcdim(P)$ unless $\vcdim(P)=2$
and $\se(P)=1$, see Observation~\ref{obs:sevcdim}.
\end{enumerate}
 For readers who are familiar with 
concepts of dimension \emph{and} convex geometries,
we state here the results of this paper.  Motivation, 
definitions, and essential preliminary material will be provided in 
Sections~2 and~3.  Proofs are given in Sections~4 through~8, and we close with
some comments on open problems that remain in Section~9.

Our first result separates Dushnik-Miller dimension and convex dimension.

\begin{theorem}\label{thm:sep-dim-cdim}
  If $n\ge3$, there is a convex geometry $P_n$ such that
  $\dim(P_n)=3$, and $\cdim(P_n)=n+1$.
\end{theorem}

The next result shows that the bound on Dushnik-Miller dimension in the preceding
theorem cannot be improved.

\begin{theorem}\label{thm:dimle2}
  If $P$ is a convex geometry and $\dim(P)\le2$, then $\cdim(P)=\dim(P)$.
\end{theorem}

Although there are convex geometries with $\VC$-dimension~$2$ and standard
example number~$1$, we show that these two parameters are essentially the
same.

\begin{theorem}\label{thm:vcdim=se}
  If $P$ is a convex geometry, then $\vcdim(P)=\se(P)$ unless $\vcdim(P)=2$ and
  $\se(P)=1$.
\end{theorem}

The next result collapses the exceptional case in the preceding theorem.

\begin{theorem}\label{thm:se=1}
  If $P$ is a convex geometry and $\se(P)=1$, then $\cdim(P)\le2$.
\end{theorem}

The next result separates convex dimension, Dushnik-Miller dimension, 
Boolean dimension and local dimension from $\VC$-dimension and 
fractional dimension.  In stating this result, we use the abbreviation 
$[n]$ for $\{1,\dots,n\}$ for an integer $n$, where for later purposes in the paper we set $[n]=\emptyset$ if $n<1$.  Also, we denote the base~$2$ logarithm of $n$ 
by $\lg n$, while the natural logarithm of $n$ is denoted $\log n$.

\begin{theorem}\label{thm:Pkn}
  If $k$ and $n$ are integers with
  $1\le k\le n-2$, there is a convex geometry $P(k,n)$ with ground set $[n]$ such that:
  \begin{enumerate}
    \item $\vcdim(P(k,n))=\se(P(k,n))=k+1$.
    \item $\fdim(P(k,n))< 2^{k+1}$.
    \item For fixed $k\ge1$,
      $\bdim(P(k,n))\rightarrow\infty$ as $n\rightarrow\infty$.
    \item For fixed $k\ge1$,
      $\ldim(P(k,n))\rightarrow\infty$ as $n\rightarrow\infty$.
    \item $\dim(P(1,n))=1+\lfloor\lg n\rfloor$ and $\dim(P(k,n))\le (k+1)2^{k+2}\log n$.
        \item $\cdim(P(k,n))=\binom{n-1}{k}$.

  \end{enumerate}
  In particular, for fixed $k\ge1$, $\dim(P(k,n))\rightarrow\infty$ as $n\rightarrow\infty$ but $\se(P(k,n))$ is constant, i.e., convex geometries are not \emph{dim-bounded}.
\end{theorem}

The most interesting and best-understood case in Theorem~\ref{thm:Pkn} is $P(1,n)$. However, we believe that it is of general interest to see $P(1,n)$ as a member of a larger family $P(k,n)$. In particular it shows that the properties of $P(1,n)$ are not an effect of low VC-dimension or standard-example number. 

%


\section{Convex Geometries and Partially Ordered Sets}\label{sec:cgs-posets}

A \textit{partially ordered set} (we prefer the 
short form \textit{poset}) is a set $P$ equipped with a binary
relation $\le$ that is reflexive, antisymmetric and transitive.
We will assume that readers are familiar with basic concepts for posets 
including: covering relation, order diagrams (also called Hasse diagrams); chains and antichains;
subposets; maximal and minimal points; isomorphic posets; and 
linear extensions.  We will also assume that readers are familiar with
the basics of finite lattices, including zeroes, ones, meets and joins.

In~\cite{EdeJam87}, Edelmann and Jamison introduce a class of posets they
call \textit{convex geometries}.  As noted in~\cite{EdeJam87}, posets in 
this class have been studied by various authors, and many combinatorial 
objects naturally carry the order structure of a convex geometry. Examples 
include subtrees of a tree~\cite{Boul67}, convex subsets of a poset~\cite{BirBen85},
convex subgraphs of an acyclic digraph~\cite{Pfal71}, transitively oriented 
subgraphs of a transitively oriented digraph~\cite{Bjor82}, convex sets 
of oriented matroids~\cite{LasVer80}, and many more as discussed 
in~\cite{EdeJam87}.

As pointed out in~\cite{EdeJam87}, the notation and terminology in the
literature for convex geometries is not uniform.  As just one example, they
have also been called \textit{antimatroids}, e.g., by Korte, 
Lov\'asz, and Schrader~\cite{KoLoSc91}.  Accordingly, a unifying 
framework for convex geometries is developed in~\cite{EdeJam87}, and 
we will follow to a large degree their framework.

Let $X$ be a finite non-empty\footnote{Some authors allow the ground set $X$ to
be the empty set, leading to the family $\{\emptyset\}$ to be a convex geometry.
Very little of interest can be said about this special case, so we only
consider non-empty ground sets.} set.  A \textit{convex geometry} with ground 
set $X$ is a family $P$ of subsets of $X$ satisfying the following three properties:

\begin{enumerate}
  \item $\emptyset, X\in P$. \hfill (\textit{Base Property})
  \item If $A$ and $B\in P$, then $A\cap B\in P$. \hfill    (\textit{Intersection Property})
  \item If $A\in P$, and $A\neq X$, there exists
    $\alpha\in X-A$ such that $A\cup\{\alpha\}\in P$.\\
    \hfill (\textit{Extension Property})
\end{enumerate}
We refer the reader to~\cite{EdeJam87} and~\cite{EdeSak88} for additional background
information on convex geometries, including several equivalent definitions.  

There are two special cases of convex geometries which are of particular interest in this 
paper. As a first example, when $L$ is a linear order on a finite non-empty set $X$, we obtain 
a convex geometry $P$ from $L$, with $|P|=1+|X|$, by taking $P$ as the set of all 
initial segments of $L$, i.e., if $|X|=n$ and $L=\alpha_1<\dots<\alpha_n$, then 
$P$ consists of the empty set together with the subsets of $X$ of the form 
$\{\alpha_1,\dots,\alpha_i\}$, where $i\in [n]$.  These simple 
examples are called \textit{linear geometries}.

As a second example, when $X$ is a finite non-empty set, and $P$ consists of \emph{all} 
subsets of $X$, then $P$ is a convex geometry.  This special case 
has been called various names in the literature, with \textit{Boolean algebra} 
and \textit{subset lattice}
being two popular choices.  To emphasize the relationship between $|X|$ and $|P|$,
we will denote this special case as $P=\bftwo^n$, where $n=|X|$.  

As a family of sets, a convex geometry is partially ordered by inclusion, and
is therefore a poset.  When $P$ is a convex geometry, $P$ is closed under 
intersections, but in general, $P$ is not closed under unions.  Nevertheless,  
a convex geometry is a lattice.  The empty set is the zero, while
the ground set $X$ is the one. For sets $A,B\in P$, we have:
\begin{align*}
  A\wedge B &= A\cap B,\quad\text{and}\\
  A\vee B &=\cap\{C\in P:A\cup B\subseteq C\}.
\end{align*}
Although we will reference results that involve posets in general, our
particular focus is on the class of convex geometries.

\section{Concepts of Dimension}\label{sec:dimension-concepts}

When $P$ is a poset, we will sometimes use the short form $a\le_P b$ as a substitute
for $a\le b$ in $P$.  Also, we will write $a\parallel_P b$ when $a$ and $b$ 
are distinct incomparable elements of $P$.
As our emphasis is combinatorial in the main, and all parameters we study have
the same value for two posets that are isomorphic, we will say that $P=Q$ 
when $P$ and $Q$ are isomorphic posets.  In the same spirit, we say $P$ 
\textit{contains} $Q$ when there is a subposet of $P$ that is isomorphic to 
$Q$.  

As it serves to motivate the definitions for the subclass of convex geometries,
we elect to begin with four concepts of dimension for the broader class
of posets.  The first concept of dimension we will investigate is the classic parameter
defined by Dushnik and Miller~\cite{DusMil41}.  
When $t\ge1$, a sequence $(L_1,\dots,L_t)$ of linear extensions of a 
poset $P$ is a \textit{realizer of $P$} if for all $x,y\in P$, 
$x\le_P y$ if and only if $x\le y$ in $L_i$ for all $i\in[t]$.  The 
\textit{Dushnik-Miller dimension} of $P$, denoted $\dim(P)$, is the least 
positive integer $t$ such that $P$ has a realizer of size~$t$.  Following the
traditions of the extensive literature on this parameter, it will henceforth
simply be called \textit{dimension}.

For each $t\ge2$, the \emph{standard example $S_t$} is a height~$2$ 
poset with minimal elements 
$a_1,\dots,a_t$, maximal elements $b_1,\dots,b_t$, and order relation $a_i<b_j$ in 
$S_t$ when $1\le i,j\le t$ and $i\neq j$.  As noted in~\cite{DusMil41},
$\dim(S_t)=t$ for all $t\ge2$.  The \textit{standard example number} of a 
poset $P$, denoted $\se(P)$, is set to be~$1$ when $P$ does not contain the 
standard example $S_2$; otherwise $\se(P)$ is the largest $t\ge2$ for which 
$P$ contains the standard example $S_t$.  Evidently, $\dim(P)\ge\se(P)$ for 
all posets $P$. 

On the one hand, the inequality $\dim(P)\ge\se(P)$ can be far from tight, as among
the class of posets which have standard example number~$1$ (this is the well
studied class of interval orders~\cite{Fis70,Trot97}), there are posets that have arbitrarily large
dimension.  At the other extreme, when $P=\bftwo^n$ is a Boolean algebra and $n>2$, then  we have $n=\dim(P)=\se(P)$. Indeed, $S_n$ is induced in $P$ by all subsets of size $1$ and $n-1$. On the other hand for $i\in [n]$ define $L_i$ as the lexicographic ordering of $\bftwo^n$ with respect to the ordering $i<i+1<\cdots<n<1<\cdots<i-1$, then it is easy to see that $(L_1,\dots,L_n)$ is a realizer of $P$. 

A class $\bbP$ of posets is said to be \emph{$\dim$-bounded} if there is a function
$f:\bfN\rightarrow \bfN$ such that for every $P\in\bbP$, $\dim(P)\le f(\se(P))$.
We point out the analogous problem in graph theory.  Although there
are triangle-free graphs with arbitrarily large chromatic number, there
are interesting classes of graphs where the chromatic number is bounded in terms
of maximum clique size.  Such classes are said to be \textit{$\chi$-bounded}. 

Blake, Hodor, Micek, Seweryn and Trotter~\cite{BHMST23+} have just resolved a 
conjecture that is more than~$40$ years old by showing that the class of posets 
that have planar cover graphs is $\dim$-bounded.  In time, readers will 
sense how their result prompts many of the questions we address in this paper.
For readers who are new to the concept of dimension for
posets, compact summaries are given in several recent research papers including
\cite{JMMTWW16,MicWie17,BlMiTr23+}, and the survey 
paper~\cite{Trot19}.  

When $P$ is a poset, we let
$\Inc(P)$ denote the set of all pairs $(a,b)$ of distinct elements of $P$ with
$a\parallel_P b$. Trivially, the following statements are equivalent:
(1)~$\Inc(P)=\emptyset$; (2)~$P$ is a chain; and (3)~$\dim(P)=1$.
When $\Inc(P)\neq\emptyset$, and $t\ge2$, a sequence $(L_1,L_2,\dots,L_t)$ of 
linear extensions of $P$ is a realizer of $P$ if and only if 
for every $(a,b)\in\Inc(P)$, there is some $i\in[t]$ such that $a>b$ in $L_i$.

When $S$ is a non-empty subset of $\Inc(P)$, we say $S$ is \textit{reversible}
if there is a linear extension $L$ of $P$ such that $a>b$ in $L$ whenever $(a,b)\in S$.
Accordingly, when $\Inc(P)\neq\emptyset$, $\dim(P)$ is the least integer $t\ge2$ such
that $\Inc(P)$ can be covered by $t$ reversible sets.

Now let $P$ be a poset that is not a chain.  A pair $(a,b)\in\Inc(P)$ is
called a \textit{critical pair} if (1)~$x<_P b$ whenever $x<_P a$; and (2)~$a<_P y$ whenever
$b<_P y$.  The set of all critical pairs of $P$ is denoted $\crit(P)$.
A sequence $(L_1,\dots,L_t)$ of linear extensions of
$P$ is a realizer of $P$ if and only if for every $(a,b)\in\crit(P)$, there
is some $i\in[t]$ such that $a>b$ in $L_i$.  Accordingly, $\dim(P)$ is the least
integer $t\ge2$ such that $\crit(P)$ can be covered with $t$ reversible sets.

Again, let $P$ be a poset that is not a chain, and let $k\ge2$.  A sequence 
$((a_1,b_1),\dots,(a_k,b_k))$ of pairs from $\Inc(P)$ is called an \textit{alternating
cycle} in $P$ if $a_i\le_P b_{i+1}$ for all $i\in[k]$.  As suggested by the terminology,
this requirement holds cyclically, i.e., we also require $a_k\le_P b_1$.
An alternating cycle $((a_1,b_1),\dots,(a_k,b_k))$ is said to be \textit{strict}
if for all $i,j\in[k]$, $a_i\le_P b_j$ if and only if $j=i+1$.
As is also well known, a non-empty subset $S\subseteq\Inc(P)$ is reversible if and only if there are no strict alternating cycles in $P$ consisting entirely of pairs from $S$.
Also, we note that when $((a_1,b_1),\dots,(a_k,b_k))$ is a strict alternating
cycle, the sets $\{a_i:i\in[k]\}$ and $\{b_i:i\in[k]\}$ are both $k$-element
antichains.

Let $P$ be a poset, and let $\cgR=(L_1,\dots,L_t)$ be a sequence of linear orders on the
ground set of $P$ (these linear orders need not be linear extensions).  For a pair
$(x,y)$ of distinct elements of $P$, determine a $0$--$1$ \emph{query string} 
$q=q(x,y,\cgR)$ of length~$t$ by setting $q(i)=1$ 
if $x<y$ in $L_i$; otherwise, set $q(i)=0$.  The sequence $\cgR$ is called a \emph{Boolean 
realizer for $P$} if there is a set $\tau$ of bit strings of length~$t$ such that 
for each pair $(x,y)$ of distinct elements of $P$, $x<_P y$ if and only if 
$q(x,y,\cgR)\in \tau$.  As defined by Ne\v{s}et\v{r}il and Pudl\'{a}k 
in~\cite{NesPud89}, the \emph{Boolean dimension} of $P$, denoted $\bdim(P)$, is 
the least positive integer $t$ for which $P$ has a Boolean realizer of size~$t$.
For every poset $P$, we always have $\bdim(P)\le\dim(P)$, 
since a realizer $\cgR=(L_1,\dots,L_t)$ is a Boolean realizer with
$\tau=\{q\}$, where $q$ is a bit string of 
length~$t$ with $q(i)=1$ for all $i\in[t]$.  On the other hand, it is an easy 
exercise to verify that if $t\ge2$ and $S_t$ is a standard example, then 
$\bdim(S_2)=2$; $\bdim(S_3)=3$; and $\bdim(S_t)=4$ for all $t\ge4$.  
We refer readers to~\cite{MeMiTr20}, \cite{FeMeMi20} and \cite{TroWal17}
for recent results on Boolean dimension.

Let $P$ be a poset.  A linear order $M$ is called a \emph{partial linear extension},
abbreviated $\ple$, of $P$ if $M$ is a linear extension of a subposet of $P$. 
A sequence $(M_1,\dots,M_m)$ of $\ple$'s of $P$ is called a \emph{local realizer of $P$}
if (1)~whenever $x\le_P y$, there is some $i\in[m]$ with $x\le y$ in $M_i$; and (2)~whenever
$(x,y)\in\Inc(P)$, there is some $j\in[m]$ with $x>y$ in $M_j$.  As defined
by Ueckerdt\footnote{T. Ueckerdt proposed the notion of local dimension at the 
2016 Order and Geometry Workshop held in Gu\l towy, Poland.},
the \emph{local dimension} of $P$, denoted $\ldim(P)$, is the
least $r$ for which there is a local realizer $(M_1,\dots,M_m)$ of $P$ such that
for each $x\in P$, there are at most $r$ different values of $i\in[m]$ with $x\in M_i$.
Again, it is clear that $\ldim(P)\le\dim(P)$ since a realizer is
a local realizer.  Also, 
it is an easy exercise to show that if $t\ge2$ and $S_t$ is the standard example,
then $\ldim(S_2)=2$ and $\ldim(S_t)=3$ for all $t\ge3$.  
We refer readers to~\cite{BoGrTr20}, \cite{KMMSSUW20} and~\cite{BPST20}
for recent results on local dimension.  In particular, \cite{BPST20} contrasts the three
notions of dimension we have defined thus far.

It is worth noting that neither of Boolean dimension and local dimension is bounded in 
terms of the other, as the following results are proved in~\cite{TroWal17}:
\begin{enumerate}
  \item  There is a constant $C$ such that if $\ldim(P)=3$, then $\bdim(P)\le C$.
  \item  For every $t\ge4$, there is a poset $P$ with $\ldim(P)=4$ and $\bdim(P)\ge t$.
  \item  For every $t\ge3$, there is a poset $P$ with $\bdim(P)=3$ and $\ldim(P)\ge t$.
\end{enumerate}

Let $P$ be a poset, and let $\cgE(P)$ be the set of all linear extensions
of $P$.  A \emph{fractional realizer of $P$} is a function $f$ which assigns
non-negative real numbers to the linear extensions in $\cgE(P)$ such that whenever
$(x,y)\in\Inc(P)$, $\sum\{f(L):L\in\cgE(P), x>y$ in $L\}\ge1$.  
As defined by Felsner and Trotter~\cite{FelTro94},
the \emph{fractional dimension} of $P$, denoted $\fdim(P)$, is the least positive
real number $t$, with $t\ge1$, such that there is a fractional realizer $f$ with
$\sum\{f(L):L\in \cgE(P)\}=t$.  Now we have $\fdim(P)\le\dim(P)$ as we can take
$f$ to the $0$--$1$ function assigning value $1$ to linear extensions in a 
realizer of $P$, while assigning value $0$ to all other linear extensions.
It is an easy exercise to verify that $\fdim(S_t)=t$ for all $t\ge2$.  On the
other hand, it is a nice exercise to show that if $\se(P)=1$, then $\fdim(P)<4$. 
We refer readers to~\cite{FelTro94} and~\cite{BiHaPo13} for additional information 
and results on fractional dimension. 

For the two special cases of convex geometries discussed in the preceding section,
we first observe that if $P$ is a linear geometry, then $\dim(P)=
\bdim(P)=\ldim(P)=\fdim(P)=\se(P)=1$.  The situation with Boolean algebras is more
complex.  If $P=\bftwo^n$ is a Boolean algebra, then
\begin{enumerate}
  \item $\dim(P)=\fdim(P)=n$, see~\cite{LR99},
  \item $\se(P)=n$ unless $n=2$, and in this case $\se(P)=1$,
  \item $\bdim(P)=\Omega(n/\log n)$ and $\bdim(P)<n$, see~\cite{BHLMM23}.
  \item $\ldim(P)= \Omega(n/\log n)$ and it is open if $\ldim(P)<n$, see~\cite{KMMSSUW20}.
\end{enumerate}

\subsection{$\VC$-Dimension}

We now discuss two additional dimension parameters defined for convex geometries 
but not for posets in general.  The first of the two is defined for any non-empty 
family of sets.

Let $X$ be a finite set, and let $F$ be a non-empty family of subsets of $X$. 
The \emph{$VC$-dimension} of $F$, denoted $\vcdim(F)$, is defined as follows.
Set $\vcdim(F)=0$ if there is no element $\alpha\in X$ for which there
are sets $A,A'\in F$ with $\alpha\in A$ and $\alpha\not\in A'$;
otherwise, set $\vcdim(F)$ to be the largest $t\ge1$ for which there
is a $t$-element subset $\{\alpha_1,\dots,\alpha_t\}$ of $X$ such that
for every set $S\subset[t]$, there is a set $A=A(S)\in F$ such that
$\alpha_i\in A$ if and only if $i\in S$.   We note that
if $P$ is a linear convex geometry, then $\vcdim(P)=1$.  Also, if
$P$ is the Boolean algebra $\bftwo^n$, then $\vcdim(P)=n$.

The next two theorems, both proved in~\cite{EdeJam87}, make clear the 
essential role Boolean algebras play in a discussion of convex geometries. For the statement, recall that for elements $x,y$ of a poset $P$ their \emph{interval} is $[x,y]=\{z\in P\mid x\leq z \leq y\}$. 

\begin{theorem}[Boolean property]\label{thm:BA}
  Let $P$ be a convex geometry, let $X$ and $Y$ be distinct elements of 
  $P$ such that (1)~$Y\neq\emptyset$; and (2)~$X$ is the intersection of all 
  sets in $P$ that are covered by $Y$ in $P$.  Then the interval $[X,Y]$ of 
  $P$ is isomorphic to the Boolean algebra $\bftwo^m$, where $m$ is the number 
  of sets that are covered by $Y$ in $P$.
\end{theorem}

Now let $Q$ be a finite lattice.  It is natural to ask whether there is a
convex geometry $P$ such that $P=Q$.  The next theorem, proved
in~\cite{EdeJam87}, provides an answer.

\begin{theorem}\label{thm:meet-dist}
  Let $Q$ be a finite lattice.  Then there is a convex geometry $P$ such
  $P$ is isomorphic to the poset $Q$ if and only if for every 
  $y$ in $Q$ with $y\neq 0$, if $x$ is the meet of all elements covered by $y$ 
  in $Q$, then the interval $[x,y]$ in $Q$ is isomorphic to the Boolean algebra 
  $\bftwo^m$, where $m$ is the number of elements of $Q$ that are covered by $y$.
\end{theorem}
In the literature, a finite lattice $Q$ is said to be \emph{meet-distributive} if it
satisfies the property in Theorem~\ref{thm:meet-dist}.  Accordingly, the study
of convex geometries can be recast as the study of meet-distributive lattices.
One useful property of these lattices is that they are \emph{graded}, i.e., all 
maximal chains between two comparable elements have the same length.

Motivated by the preceding discussion, we make the 
following definitions.  Let $Q$ be a finite poset, and let $y\in Q$.
We denote by $\dd(y,Q)$ the \textit{down degree} of $y$ in $Q$, i.e.,
$\dd(y,Q)$ is the number of elements of $Q$ that are covered by $y$ in $Q$.  
Of course $\dd(y,Q)=0$ if and only if $y$ is a minimal element of $Q$.
In turn, we let $\maxdd(Q)$ denote the maximum value of $\dd(y,Q)$ taken over all 
elements $y\in Q$.  

When $P$ is a convex geometry, if $d=\maxdd(P)$, then 
Theorem~\ref{thm:BA} now implies that $P$ contains the Boolean algebra $\bftwo^d$.
Therefore, $\dim(P)\ge\vcdim(P)\ge\maxdd(P)$.  The next theorem, proved
in~\cite{BCDK06} shows that this inequality is tight.

\begin{theorem}\label{thm:vcdim=maxdd}
  When $P$ is a convex geometry, $\vcdim(P)=\maxdd(P)$.
\end{theorem}
Because some of our proofs make direct use of the diagram and the specific 
value of maximum down degree and for the balance of the paper we make the following choice:
\begin{quote}
From now on we will phrase results in terms of maximum down
degree instead of $\VC$-dimension,
                               asking readers to keep in mind that when working with convex geometries,
maximum down degree is \emph{exactly} the same as $\VC$-dimension.  
\end{quote}

Dually, when $Q$ is a finite poset, and $x\in Q$, we define the
\textit{up degree} of $x$ in $Q$, denoted $\ud(x,Q)$ as the 
number of elements of $Q$ that cover $x$.  Also, $\maxud(X)$ is the maximum value
of $\ud(x,Q)$ taken over all $x\in Q$.  As we will soon see, there is
no bound on the value of maximum up degree, even when maximum down degree is~$2$.

We finish with a quick observation announced in the introduction:
\begin{observation}\label{obs:sevcdim} Let $P$ be a convex geometry, then
 $\se(P)\ge\maxdd(P)$ unless $\maxdd(P)=2$
and $\se(P)=1$.
\end{observation}
\begin{proof}
 If $\maxdd(P)=d\geq 3$, then by Theorem~\ref{thm:BA} there is a Boolean algebra $\mathbf{2}^d$ contained in $P$. Hence, by $d=\se(\mathbf{2}^d)\leq \se(P)$. 
\end{proof}

\subsection{Convex Dimension}

The second dimension parameter for convex geometries borrows from the
set up for (Dushnik-Miller) dimension.
When $X$ is the ground set of a convex geometry $P$, $|X|=n$, and 
$L=\alpha_1<\dots<\alpha_n$ is a linear order on $X$, we say $L$ is 
\textit{compatible} when $\{\alpha_1,\dots,\alpha_i\}$ is a set in $P$ for each $i\in[n]$.  
There is a natural $1$--$1$ correspondence between compatible linear orders 
on $X$ and maximal chains in $P$.

Let $X$ be a finite set, let
$t\ge1$, and let $(P_1,\dots,P_t)$ be a sequence of convex geometries
each of which has ground set $X$. 
Define a family $P$ of subsets of $X$ by:
\[
  P=\{A_1\cap\dots\cap A_t:A_i\in P_i \text{ for all $i\in[t]\}$.}
\]
Then $P$ is a convex geometry, and we denote this by writing
$P=P_1\vee\dots\vee P_t$.  Furthermore, given a convex geometry $P$, if 
we define the sequence $(P_1,\dots,P_t)$ by taking the linear geometries 
associated with the maximal chains in $P$, then
$P=P_1\vee\dots\vee P_t$.  These observations give rise to the
following definitions.

Let $P$ be a convex geometry, and let $X$ be the ground set of $P$. 
A sequence $(P_1,\dots,P_t)$ of linear geometries, each with ground set $X$, is
called a \textit{convex realizer of $P$} if $P=P_1\vee\dots\vee P_t$.
The \textit{convex dimension} of $P$,
denoted $\cdim(P)$, is the least positive integer $t$ for which $P$ has
a convex realizer $(P_1,\dots,P_t)$ of size~$t$.

We pause to make an elementary observation.  Let $n\ge2$, let
$L_1=1<2<\dots<n$, and $L_2=n<\dots<2<1$.  If $Q_1$ and $Q_2$ are the
linear geometries with ground set $[n]$ determined by $L_1$ and $L_2$,
respectively, and $P=Q_1\vee Q_2$, then $\cdim(P)=\dim(P)=\maxdd(P)=2$,
while $\maxud(P)=n$.

In a finite lattice $Q$, an element $y$ that satisfies $\ud(y,Q)=1$ is called
a \emph{meet-irreducible element}.  This terminology reflects the property that
if $y=w\wedge z$, then one of $w$ and $z$ must be $y$.  
The following theorem, which provides a \emph{very} useful alternative
definition of convex dimension, is proved in~\cite{EdeJam87}.  

\begin{theorem}\label{thm:cdim-width}
  Let $P$ be a convex geometry.  Then $\cdim(P)$ is the width of the
  subposet of $P$ determined by the meet-irreducible elements of $P$.
\end{theorem}

Next, we present a brief series of elementary results that support more substantive
arguments to follow. 

\begin{proposition}\label{pro:crit-pair-dd-ud}
  Let $Q$ be a finite lattice. If $(x,y)\in\crit(Q)$,
  then $\dd(x,Q)=\ud(y,Q)=1$.
\end{proposition}

\begin{proof}
  We prove that $\dd(x,Q)=1$.  The argument to show that $\ud(y,Q)=1$ is
  dual.  Suppose to the contrary that $\dd(x,Q)\ge2$.
  Then there are distinct elements $w$ and $z$ of $Q$ such that $x$ covers both 
  $w$ and $z$ in $Q$.
  Since $(x,y)$ is a critical pair, and $w<_Q x$, we know $w<_Q y$.
  Similarly, we have $z<_Q x$, so $z<_Q y$.  Since $Q$ is a lattice
  and $\{x,y\}$ is an antichain, we know that $\{w,z\}\leq_Q w\vee z\leq_Q x \wedge y <_Q\{x,y\}$.
  However, these inequalities imply that neither of $w$ and $z$ is covered by $x$.
  The contradiction completes the proof.
\end{proof}

When $Q$ is a finite lattice, and $y$ is a meet-irreducible
element of $Q$, we let $\uc(y,Q)$ denote the unique element of $Q$ that covers~$y$.
An element $x$ of $Q$ is called a \textit{join-irreducible element of $Q$} if
$\dd(x,Q)=1$.  In this case, we let $\dc(x,Q)$ denote the unique element of $Q$
covered by $x$.  The next result shows that in a convex geometry $P$, there is a 
natural $1$--$1$ correspondence between the set of meet-irreducible elements of 
$P$ and the set of critical pairs of $P$.

\begin{proposition}\label{pro:J-critP}
  Let $P$ be a convex geometry, let $B$ be a meet-irreducible element of $P$, let $Y=
  \uc(B,P)$, and let $\{\alpha\}=Y-B$.  Then the following statements are
  equivalent:
  \begin{enumerate}
    \item $A=\cap\{U\in P:\alpha\in U\}$.
    \item $(A,B)\in\crit(P)$.
  \end{enumerate}
\end{proposition}

\begin{proof}
  Assume that  $A=\cap\{U\in P:\alpha\in U\}$.  We show that $(A,B)$ is
  a critical pair.  First, since $\alpha\in A$, and $\alpha\not\in B$,
  we know $A\not\subset B$.  If there is an element $\beta\in A-Y$, then
  $A\cap Y\subsetneq A$, which contradicts the definition of $A$.  We conclude
  that $A\subset Y$.  Furthermore, if $V\subsetneq A$ for some $V\in P$, then $\alpha\not\in V$.  
  Since $Y=B\cup\{\alpha\}$, we conclude that $V\subset B$.  Now
  suppose that $B$ is a proper subset of some $W\in P$.  Then since $B$ is a meet-irreducible element, we have $Y\subseteq W$.  This
  implies $A\subset W$.  These observations imply that $(A,B)$ is a critical
  pair.

  Now suppose that $(A,B)$ is a critical pair.  We show that
  $A=\cap\{U\in P:\alpha\in U\}$.  Since $B\subsetneq Y$, we must have
  $A\subseteq Y$.  However, $Y=B\cup\{\alpha\}$, and we conclude that
  $\alpha\in A$.   Now suppose that there is a set $U\in P$ with
  $\alpha\in U$ and $U\subsetneq A$.  Then $U\not<_P B$, which contradicts
  the assumption that $(A,B)$ is a critical pair.  We conclude
  that $A=\cap\{U\in P:\alpha\in U\}$.
\end{proof}

\begin{proposition}\label{pro:cg-se}
  Let $Q$ be a finite poset.  If $t=\se(Q)\ge2$, then $Q$ contains
  a copy of the standard example $S_t$ labeled such that $(a_i,b_i)\in\crit(Q)$
  for all $i\in[t]$.
\end{proposition}

\begin{proof}
  Of all copies of the standard example $S_t$ contained in $Q$,  choose one
  for which the following sum is minimum:
  \[
    \sum_{i\in[t]} |D_Q[a_i]|+|U_Q[b_i]|.
  \]
  Clearly, $(a_i,b_i)$ is a critical pair in $Q$, for each $i\in[t]$.
\end{proof}

\begin{proposition}\label{pro:cdim-dim}
  Let $P$ be a convex geometry.  Then 
  \[
    \cdim(P)\ge\dim(P)\ge\max\{\maxdd(P),\se(P),\bdim(P),\fdim(P),\ldim(B)\}.
  \]
\end{proposition}

\begin{proof}
  As noted in Section~\ref{sec:dimension-concepts}, the inequalities $\dim(P)\ge\se(P)$, $\bdim(P)$, $\fdim(P)$, $\ldim(B)$ hold for every poset $P$. For $\dim(P)\ge\maxdd(P)$ note that by Theorem~\ref{thm:meet-dist} the poset $P$ contains a Boolean algebra of dimension $\maxdd(P)$. Let now $J$ be the
  set of meet-irreducible elements in $P$, i.e., $J$ is the set of all sets $B$ for
  which $\ud(B,P)=1$. From Theorem~\ref{thm:cdim-width}, we know that $\cdim(P)$ is the 
  width of the subposet $J$. 
  If $\cdim(P)=s$, it follows that there are chains
  $\cgC_1,\dots,\cgC_s$ in $P$ that cover all the sets in $J$.  For each
  $i\in [s]$, let $S_i$ consist of all critical pairs $(A,B)$ such that $B\in\cgC_i$.
  Each set $S_i$ of this form is reversible since it cannot contain a strict 
  alternating cycle.  It follows that $\dim(P)\le s$.
\end{proof}

\section{Separating Dimension and Convex Dimension}\label{sec:sep-dim-cdim}

Proposition~\ref{pro:cdim-dim} asserts that if $P$ is a convex geometry, then 
$\cdim(P)\ge\dim(P)$.  If $P$ is a linear geometry or the Boolean algebra $\bftwo^n$,
then $\cdim(P)=\dim(P)$.  More generally, this equality even holds for distributive       lattices, see~\cite{EdeSak88}.  However, the inequality is not always tight:
Figure~1 in~\cite{EdeSak88} illustrates a convex geometry with convex dimension
$4$ but (Dushnik-Miller) dimension $3$.

It is natural to ask whether dimension and convex dimension 
can be \emph{separated}, i.e., is there a family of convex geometries for
which dimension is bounded while convex dimension is not. 
We now proceed to give an affirmative answer. The remainder of this section
constitutes the proof of Theorem~\ref{thm:sep-dim-cdim}.

For $n\geq 3$ denote by $Q_n$ the sets of the form $[i]\times[j]\times[k]$ for $i\in[2]$ and $j,k\in [n]$. It is easy to see that the dimension of $Q_n$ is at most $3$. Pick lexicographic orderings $L_1, L_2, L_3$ of $Q_n$ where each orders $Q_n$ lexicographically with respect to a different cyclic ordering of the three coordinates.

Denote by $P_n$ the subposet of $Q_n$ consisting of the sets of the form $[2]\times[j]\times[k]$ for all $j,k\in [n]$ but $[1]\times[j]\times[k]$ only for those $j,k\in [n]$ such that $j+k\leq n+1$. As a subposet of $Q_n$, also $\dim(P_n)\leq 3$.

Note that $P_n$ contains the empty set as well as the full set $[2]\times[n]\times[n]$. Further, it clearly is intersection closed and every element can be either extended by $2$ or otherwise by any other element $\alpha$ missing on the first or second coordinate. We conclude that $P_n$ forms a convex geometry. Since $\maxdd(P_n)=3$,
we know $\dim(P_n)\ge 3$ by Proposition~\ref{pro:cdim-dim}. Therefore, $\dim(P_n)=3$.

Finally, note that $u(Y,P_n)=1$ exactly for the subsets of the form $[2]\times[n]\times[k]$, $[2]\times[k]\times[n]$ for $1\leq k<n$ as well as $[1]\times[j]\times[k]$ for $j+k=n+1$. And furthermore the latter type of sets forms a maximum antichain of size $n+1$. Theorem~\ref{thm:cdim-width} we have $\dim(P_n)=n+1$. 

We conclude that we have an infinite family $\{P_n:n\ge3\}$ of convex geometries
with $\dim(P_n)=3$, and $\cdim(P_n)= n+1$.  Accordingly, we have
separated dimension and convex dimension for the class of
convex geometries.

%

We show in 
Figure~\ref{fig:Knauer-2} an order diagram for $P_6$ as a subposet of $Q_6$. 

\begin{figure}
  \begin{center}
    \includegraphics[scale=1.8]{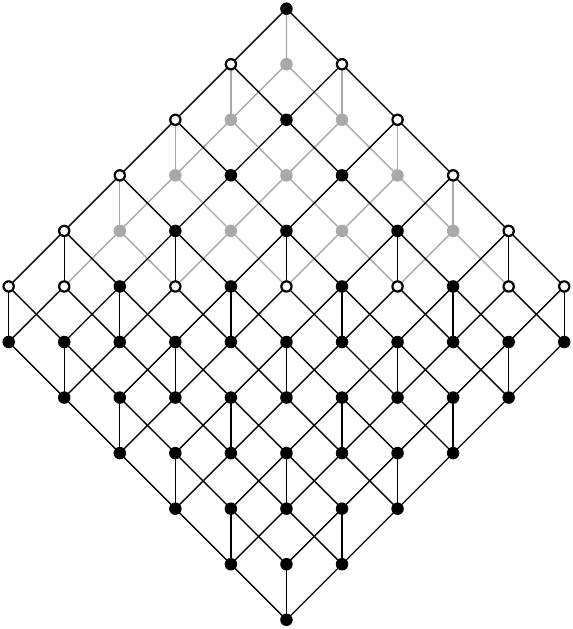}

  \end{center}
  \caption{For $n\ge3$, we illustrate a $3$-dimensional poset 
  $Q_n$ and its subposet $P_n$  obtained by removing the
  grey points. Then $P_n$ is a convex geometry, and the white points constitute the set of meet-irreducible elements of $P_n$.}
  \label{fig:Knauer-2}
\end{figure}
%
%

\section{Convex Geometries with Dimension~$2$}\label{sec:dimle2}

The family constructed in the preceding section suggests the
following question. Among convex geometries with 
dimension~$2$, is convex dimension bounded?  In this section,
we give an affirmative answer.  In fact, we show that when $\dim(P)\le2$,
$\cdim(P)=\dim(P)$, and the argument constitutes the proof 
of Theorem~\ref{thm:dimle2}.  The statement holds trivially if $\dim(P)=1$, so
we fix a convex geometry $P$ with $\dim(P)=2$. 

Our argument requires an unpublished
but by now quite well known result due to~K.\ Baker:
A finite lattice has Dushnik-Miller dimension at most~$2$ if and only if
its order diagram is planar.  Recall that a poset is said to be planar
if its order diagram can be drawn (following all rules for diagrams) without
edge crossings in the plane.  We show on the left side of 
Figure~\ref{fig:Knauer-5} a $2$-dimensional convex geometry with a planar 
order diagram. 

\begin{figure}
  \begin{center}
    \includegraphics[scale=1.8]{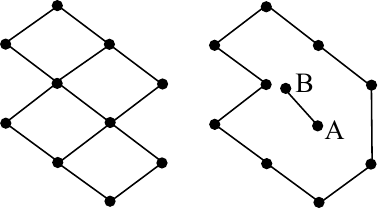}

  \end{center}
  \caption{On the left, we show a convex geometry with a planar order diagram.
  Note that maximum down degree is~$2$, and every interior face is a diamond.
  On the right, we suggest how an element in the interior would appear
  if it had up degree~$1$.  We will show that this is impossible.}
  \label{fig:Knauer-5}
\end{figure}

Using the result of Baker, we take a planar drawing of the order diagram of 
$P$.  Since $P$ is a lattice, there is a well defined maximal chain in $P$ that
constitutes the left boundary in the drawing. Also, there is a well defined
maximal chain in $P$ that constitutes the right boundary in the two drawings.
These two boundary chains can intersect. Regardless, any element of $P$ that is
not on either of these two chains is in the \emph{interior} of the drawing.

\begin{claim}
  The interior faces in the drawing are diamonds. 
\end{claim}

\begin{proof} 
  Since $P$ is a lattice, each face $F$ has a unique maximum element which
  we denote $1_F$.  Also, $F$ has a unique minimum element which we denote
  $0_F$. Let $A,B$ be the vertices of $F$ covered by $1_F$. By the 
  Boolean property, we know that element $A\cap B$ is covered by $A$ and $B$, so 
  $F$ is a diamond formed by the elements of $\{1_F,A,B,A\cap B\}$.  
\end{proof}

\begin{claim}
  If $A$ is in the interior of the drawing, then $\ud(A,P)\ge2$,
\end{claim}

\begin{proof}
  Referring to the illustration on the right side of Figure~\ref{fig:Knauer-5},
  suppose to the contrary that $A$ is in the interior of the drawing, and $\ud(A,P)=1$.
  Let $B$ be the unique element of $P$ that covers~$A$.  Also, let $e$ be the edge in the
  drawing having $A$ and $B$ as its end points.  All the points in the plane that belong to
  $e$, except possibly $B$, are in the interior of the drawing.  This implies that $e$
  is a boundary edge of two interior faces $F$ and $F'$ that have no points in common 
  other than the end points of the edge $e$.

  We note that we cannot have $B=1_F$, and $B=1_{F'}$, as this would require
  $\dd(B,P)\ge3$.  Without loss of generality, we may assume $B\neq 1_F$.  Since
  $F$ is a diamond, this forces $B=0_F$. In turn, this forces $\ud(A,P)\ge2$.
  The contradiction completes the proof of the claim.
\end{proof}

To complete the proof of the theorem, we simply observe that the width of the
subposet of $P$ consisting of all elements that have up degree~$1$ in $P$ is at most~$2$, 
since all these elements are on the two boundary chains in the drawing.

\section{Standard Example Number and Maximum Down Degree}

We have already noted that when $P$ is a convex geometry with ground set
$X$, then $\se(P)\ge\maxdd(P)$, except possibly when $\maxdd(P)=2$.  If 
$\maxdd(P)=2$, we can have $\se(P)=1$, as is the case when $P$ is the 
Boolean algebra $\bftwo^2$.  This raises the question as to whether there is a 
convex geometry $P$ for which $\se(P)>\maxdd(P)$.  We answer this question in the 
negative.  The results of this section constitute the proof of Theorem~\ref{thm:vcdim=se}.

We argue by contradiction and assume that $P$ is a convex geometry with
$\se(P)>\maxdd(P)$.  Then $\se(P)\ge 2$.  Choose a standard example $S_n$ in $P$ 
with the elements labeled as $\{A_1,\dots,A_n\}\cup\{B_1,\dots,B_n\}$, so that for 
all $i,j\in[n]$, $A_i\subseteq B_j$ if and only if $i\neq j$.  By 
Proposition~\ref{pro:cg-se}, we may further assume that for each $i\in[n]$, $(A_i,B_i)$ 
is a critical pair. From Proposition~\ref{pro:crit-pair-dd-ud}, it follows that 
for each $i\in[n]$, $\dd(A_i,P)=\ud(B_i,P)=1$. 

For each $i\in[n]$, let $Y_i$ be the unique element covering $B_i$ in $P$.
Since $|Y_i-B_i|=1$, there is a unique element $\alpha_i\in X$ such that
$Y_i=B_i\cup\{\alpha_i\}$.  Since $(A_i,B_i)$ is a critical pair, and $B_i\subsetneq Y_i$,
we know $A_i\subset Y_i$.  Since $A_i\parallel_P B_i$, this forces $\alpha_i\in A_i-B_i$.

Now let $X_i$ be the unique set in $P$ that is covered by $A_i$. Since $(A_i,B_i)$
is a critical pair, we know $X_i\subset B_i$.  This now requires $X_i=A_i-\{\alpha_i\}$.
i.e., $\alpha_i$ is the only element of $A_i$ that is not in $B_i$.

Now set $Y_0=A_1\vee\dots\vee A_n$.  Also, for each $i\in[n]$, let
$B'_i=Y_0\cap B_i$.  Then $B'_i$ is an element of $P$, and $\alpha_i\in A_i-B'_i$.
Furthermore, for each $j\in[n]$ with $j\neq i$, $A_j$ is a subset of $B_i$ and
$Y_0$. Therefore, $A_j\subset B'_i$.  It follows that the
sets in $\{A_1,\dots,A_n\}\cup\{B'_1,\dots,B'_n\}$ determine a subposet of
$P$ which is isomorphic to $S_n$.  In particular, the sets in $\{B'_1,\dots,B'_n\}$
form an $n$-element antichain in $P$.

For each $i\in[n]$, let $Y'_i=B'_i\cup\{\alpha_i\}$.  We then have:
\begin{multline}\label{eqn:Y'_i}
  Y'_i = B'_i\cup\{\alpha_i\}=(Y_0\cap B_i)\cup\{\alpha_i\}=(Y_0\cup\{\alpha_i\})\cap (B_i\cup\{\alpha_i\})=Y_0\cap Y_i.
\end{multline}
Since $Y_0$ and $Y_i$ are elements of $P$, it follows that $Y'_i$ is an element of
$P$.  Also, from its definition, it is clear that $Y'_i$ covers
$B'_i$.

\begin{claim}
  If $i,j\in[n]$, then $A_j\subset Y'_i$.
\end{claim}

\begin{proof}
  If $j\neq i$, then $A_j\subset B'_i\subset Y'_i$.  If $j=i$, since
  $A_i\subset Y_i$ and $A_i\subset Y_0$, it follows from (\ref{eqn:Y'_i}) that 
  $A\subset Y'_i$.
\end{proof}

\begin{claim}
  If $i\in[n]$, then $Y_0= Y'_i$.
\end{claim}

\begin{proof}
  The first claim implies that $Y_0\subseteq Y'_i$ for all $i\in[n]$.
  If $i\in[n]$, the definition of $B'_i$ implies that $B'_i\subseteq Y_0$.  
  Since $Y'_i$ covers $B'_i$, we conclude that $Y_0=Y'_i$.
\end{proof}

We have now shown the the element $Y_0$ covers all elements of $\{B'_1,\dots,B'_n\}$.
This implies $\maxdd(P)\ge\dd(Y_0,P)\ge n$.  In turn, we have $\maxdd(P)\ge n$.
With this observation, the proof is complete.

\section{Convex Geometries with Standard Example Number $1$}\label{sec:se=1}

We now know that if $P$ is a convex geometry, then $\se(P)=\maxdd(P)$,
except possibly when $\se(P)=1$ and $\maxdd(P)=2$.  In this section,
we show that if $P$ is a convex geometry, and $\se(P)=1$, then
$\dim(P)=\cdim(P)\le 2$.  This section constitutes the proof of 
Theorem~\ref{thm:se=1}.  

We fix a convex geometry $P$ with $\se(P)=1$.  If $\maxdd(P)=1$, then $P$
is a chain, and the conclusion holds trivially.  So we may assume that
$\maxdd(P)=2$ and hence $\dim(P)\geq 2$.

Let $J$ denote the subposet of $P$ determined by the meet-irreducible
elements.  Then $\cdim(P)=\width(J)$.  If the proposition fails,
there is a $3$-element antichain $\{B_1,B_2,B_3\}$ in $J$.
For each $i\in[3]$, let $Y_i$ be the unique element of $P$ that
covers $B_i$.  Since $\maxdd(P)=2$, we may assume that $Y_1\neq
Y_2$.  Since $P$ does not contain $S_2$, either $B_1<_P Y_2$
or $B_2<_P Y_1$.   Note that since $P$ is a lattice, it cannot be the case that
both inequalities hold, so we may assume without loss of generality
that $B_1<_P Y_2$ and $B_2\parallel_P Y_1$.

Then since $Y_1$ is the unique element covering $B_1$, we have $Y_1<_P Y_2$, and there is an element $Z_1$ covered by
$Y_2$ such that $Y_1\le_P Z_1$.  The Boolean property
then implies that there is an element $X$ covered by $Z_1$ and
$B_2$.  Furthermore, since $P$ is graded and $B_1 \parallel_P B_2$,
if $Y_1<_P Z_1$, then $B_1,Y_1 \parallel_P X$, and $\{B_1,Y_1,B_2,X\}$ 
determines $S_2$.  We conclude that $Y_1=Z_1$.

We note that $\{B_3,Y_3\}\cap\{B_1,Y_1,B_2,Y_2,X\}=\emptyset$, as
any common point implies a comparability that does not exist or a down degree 
larger than two.   Furthermore, note that $Y_1,Y_2,Y_3$ must form a chain in some order: otherwise if for $i=1,2$ we would have $Y_3\parallel_P Y_i$, then since $Y_i$ covers $B_i$, we get $Y_3\parallel_P B_i$ yielding a $S_2$ together with $B_3$.
Now, since $Y_2$ covers $Y_1$ and $Y_1$ has  already down degree two, we must have 
$Y_2<Y_3$. Let $Y_2\leq Z_2< Y_3$ be covered by $Y_3$. The Boolean property then 
implies that there is an element $W$ covered by $Z_2$ and $B_3$. Since $Y_2$ already 
has down degree two, and $\{B_1,B_2,B_3\}$ is an antichain, we force $Y_2<Z_2$. 
Since $P$ is graded, $Y_2$ cannot be larger than $W$.  Since $\{B_1,B_2,B_3\}$ is 
an antichain, $B_3$ cannot be larger than $B_2$. Hence $\{Y_2,B_2,B_3,W\}$ induce a copy of $S_2$.  The contradiction completes the proof.

\section{Separating Three Parameters from the Other Four}\label{sec:Pkn} 

We have noted that when $P$ is a convex geometry, then $\dim(P)\ge\maxdd(P)$.
We note that if $P$ is the Boolean algebra $\bftwo^n$, or a linear geometry, 
then $\dim(P)=\maxdd(P)$.  Note that these parameters also agree on the convex
geometries discussed in the preceding two sections.
Now it is natural to ask whether there is a class of convex geometries on which 
maximum down degree is bounded but dimension is not.  We now proceed 
to give an affirmative answer to this question. In fact, we will construct an infinite
family of convex geometries for which (1)~maximum down degree, fractional
dimension, and standard example
number are bounded; while (2)~convex dimension,
dimension, Boolean dimension, and local dimension are unbounded.  


The remainder of this section constitutes the proof of 
Theorem~\ref{thm:Pkn}.  For the readers convenience, we give nine statements including the ones of the theorem.  Let $k$ and $n$ be integers with $1\le k\le n-2$.  
We show there is a family $P(k,n)$ of subsets of $[n]$ satisfying the
following properties:

\begin{enumerate}
  \item[1.] $P(k,n)$ is a convex geometry with ground set $[n]$.
  \item[2.] If $1\le k<k'\le n-2$, then $P(k,n)$ is a subposet of $P(k',n)$.
  \item[3.] $\cdim(P(k,n))=\binom{n-1}{k}$.
  \item[4.] $\maxdd(P(k,n))=k+1$.
  \item[5a.] $\dim(P(1,n))=1+\lfloor\lg n\rfloor$,
  \item[5b.] $\dim(P(k,n))\le (k+1)2^{k+2}\log n$.
  \item[6.] $\se(P(k,n))=k+1$.
  \item[7.] For fixed $k\ge1$,
    $\bdim(P(k,n))\rightarrow\infty$ as $n\rightarrow\infty$.
  \item[8.] For fixed $k\ge1$,
    $\ldim(P(k,n))\rightarrow\infty$ as $n\rightarrow\infty$.
  \item[9.] $\fdim(P(k,n))< 2^{k+1}$.
\end{enumerate}

We begin with the definition of $P(k,n)$. Then we proceed to prove that
the nine statements of the theorem are satisfied.  

Let $k$ and $n$ be integers with $1\le k\le n-2$.  A set $A\subseteq[n]$ belongs to
$P(k,n)$ if and only if the following property is satisfied:

\smallskip
\noindent
\textit{Membership.}\quad If $|A|=k+i-1$, then $[i-1]\subseteq A$.

\smallskip
With our convention that $[i-1]=\emptyset$ when $i-1\le0$, we note that 
$A\in P(k,n)$
whenever $A\subset[n]$ and $|A|\le k$.  On the other hand, $\{1,2,3,6,11\}$
and $\{1,2,6,10,11\}$ belong to $P(3,12)$ while $\{1,3,6,10,11\}$ does not.

The next result is statement~1.

\begin{proposition}\label{pro:CGkn}
  The family $P(k,n)$ is a convex geometry with ground set $[n]$. 
\end{proposition}

\begin{proof}
  By definition $P(k,n)$ is a family of subsets of $[n]$.  We now
  show that $P(k,n)$ satisfies the Base, Intersection and Extension
  Properties.  First, we observe that $\emptyset$ and $[n]$ satisfy the
  Membership requirement, so both are in $P(k,n)$.  
  Therefore, $P(k,n)$ satisfies the Base Property.

  Let $A,B\in P(k,n)$.  We show that $A\cap B\in P(k,n)$.  This holds
  trivially if one of $A$ and $B$ is a subset of the other, so we may
  assume that neither is contained in the other.  The conclusion that
  $A\cap B\in P(k,n)$ also holds trivially if $|A\cap B|\le k$,
  so we may assume that $|A\cap B|=k+i-1$ for some $i$ with $2\le i\le n-k+1$.  Then 
  $|A|\ge k+i-1$ and $|B|\ge k+i-1$, so $[i-1]\subset A$ and
  $[i-1]\subset B$.  Therefore, $[i-1]\subset A\cap B$.  Hence, $A\cap B\in P(k,n)$.
  With these observations, we have shown that $P(k,n)$ satisfies the
  Intersection Property.

  Finally, let $A\in P(k,n)$ with $A\neq X$.  Suppose first that $|A|\le k$.  
  If $1\not\in A$, then $A\cup\{1\}\in P(k,n)$.  If $1\in A$, then
  $A\cup\{\alpha\}\in P(k,n)$ for every $\alpha\in[n]-A$.
  Now suppose that $|A|=k+i-1$ for some $i$ with $2\le i\le n-k+1$.
  Then $[i-1]\subset A$.
  Since $A\neq[n]$, we know $i\le n-k$. If
  $i\not\in A$, then $A\cup\{i\}\in P(k,n)$. If $i\in A$,
  then $A\cup\{\alpha\}\in P(k,n)$ for every $\alpha\in[n]-A$.
  With this observation, we have shown that $P(k,n)$ satisfies the
  Extension Property.  Therefore $P(k,n)$ is a convex 
  geometry with ground set $[n]$.
\end{proof}

The subfamily $\{P(1,n):n\ge3\}$ will play an important role in
arguments to follow.  We illustrate $P(1,5)$ in 
Figure~\ref{fig:knauer-3} and note that $P(1,4)$ is the same 
convex geometry illustrated in Figure~2 in~\cite{EdeSak88}.

\begin{figure}
  \begin{center}
    \includegraphics[width=.7\textwidth]{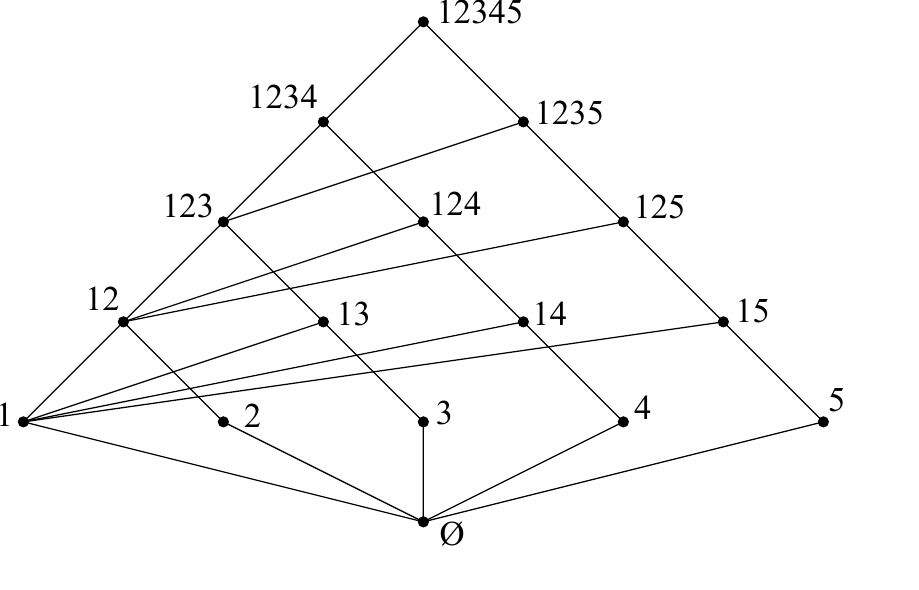}

  \end{center}
  \caption{We illustrate the convex geometry $P(1,5)$.  In the drawing, sets
  are indicated without braces and commas.  Although these facts will follow 
  from the more general arguments given below, readers may enjoy verifying that 
  $\cdim(P(1,5))=4$, $\dim(P(1,5))=\bdim(P(1,5))=\fdim(P(1,5))=3$, and 
  $\se(P(1,5))= \maxdd(P(1,5))=2$.}
  \label{fig:knauer-3}
\end{figure}

The next result is statement 2, and it helps to explain our interest in
the subfamily $\{P(1,n):n\ge3\}$.  

\begin{proposition}\label{pro:monotonic}
  If $1\le k<k'\le n-2$, then $P(k,n)$ is a subposet of $P(k',n)$.
\end{proposition}

\begin{proof}
  Let $A\in P(k,n)$ and suppose that $|A|=k'+i-1$.
  Then $|A|=k+(k'-k)+i-1$.  Since $A\in P(k,n)$, we must have
  $[k'-k+i-1]\subset A$.  Since $k'-k>0$, this implies $[i-1]\subset A$.
  Therefore, $A\in P(k',n)$.
\end{proof}

Let $J(k,n)$ denote the family of all sets of the form $[i-1]\cup B$ where
$i$ and $B$ satisfy the following requirements:
\begin{enumerate}
  \item $i\in[n]$ and $B\subset[n]$. 
  \item $i<j$ for every $j\in B$.
  \item $|B|\le k$.
  \item If $|B|<k$, then $B=\{j\in[n]: i+1\le j\le n\}$.
\end{enumerate}

\begin{proposition}\label{pro:Jkn}
  If $1\le k\le n-2$, then $J(k,n)$ is the set of all meet-irreducible
  elements of $P(k,n)$.
\end{proposition}

\begin{proof}
  Clearly, all sets in $J(k,n)$ satisfy the membership requirement and
  belong to $P(k,n)$.  Note further that all sets in $J(k,n)$ have size of  at least $k$.  

  We observe that if $A=[i-1]\cup B$ is in $J(k,n)$ and $|B|<k$.
  Then $|A|=n-1$, and $[n]$ is the only element of $P(k,n)$ that covers
  $A$.  On the other hand, if $A=[i-1]\cup B$ is in $J(k,n)$ and
  $|B|=k$, then the only element of $P$ covering $A$ is $A'=[i]\cup B$.
  With these observations, we have verified that all elements of $J(k,n)$
  are meet-irreducible.

  For the converse, suppose $A$ is a meet-irreducible element of $P(k,n)$. Clearly $n>|A|$ since the entire set is not a meet-irreducible element.   
  If $|A|=n-1$, say $A=[n]-\{i\}$, then $A=[i-1]\cup B$, where $B=\{j\in[n]:i+1\le j\le n\}$. Since the membership requirement of $P(k,n)$ forces $[n-k-1]\subset A$ we have $i\geq n-k$ and $|B|\leq k$. This implies $A\in J(k,n)$.

  If $k\le|A|\le n-2$, write $|A|=k+i-1$ where $1\le i\le n-k-1$.
  Then by the membership requirement of $P(k,n)$ we have $[i-1]\subseteq A$.  If $i\not\in A$, then $A=[i-1]\cup B$ where $|B|=k$. Hence, $A\in J(k,n)$. On the other
  hand, if $i\in A$, then $A$ is covered by $A\cup\{j\}$ for all $j\in [n]-A$.
  The assumption that $|A|\le n-2$ then implies that $\ud(A,P(k,n))\ge2$, contradicting that $A$ was a meet-irreducible element.

  Finally note that $|A|\ge k$, for if $|A|<k$, then $A$ is covered by $A\cup\{\alpha\}$ for every $\alpha\in[n-k]$. Since  by definition
  $k\leq n-2$, we get $\ud(A,P(k,n))\ge2$, contradicting that $A$ was a meet-irreducible element. Hence, we have considered all the cases and the proof is complete.
\end{proof}

For the balance of this section, when we say that $A=[i-1]\cup B$ is
meet-irreducible, we also mean that the requirements for membership
in $J(k,n)$ are satisfied by $i$ and $B$.

The next result is statement 3.

\begin{proposition}
  If $1\le k\le n-2$, then $\cdim(P(k,n))=\binom{n-1}{k}$.
\end{proposition}

\begin{proof}
  Let $J=J(k,n)$. Then set $w=\binom{n-1}{k}$. 
  The set $\cgA$ of all meet-irreducible elements of $P(k,n)$ of size $k$ is 
  an antichain in $J$.  Note that a $k$-element set $A\subset[n]$ belongs to
  $J$ if and only if $1\not\in A$.  It follows that the width of 
  $J$ is at least $|\cgA|=w$.

  To show that $\width(J)\le w$, we construct an explicit covering of $J$
  using $w$ chains.  Let $B$ be a $k$-element subset of $[n]$ with $1\not\in B$,
  and let $j=\min(B)$.  Then $j\ge2$, and the sets in
  $\{[i-1]\cup B: 1\le i\le j-1\}$ form a chain of size~$j-1$ in $J(k,n)$.
  Clearly, every element of $J(k,n)$ is contained in a unique chain of this
  form.
\end{proof}

The next result implies statement 4.

\begin{proposition}\label{pro:dd}
  Let $A$ be a non-empty set in $P(k,n)$.  Then $\dd(A,P(k,n))=1$ if and only if
  $A$ is a singleton.  Furthermore, $\maxdd(P(k,n))=k+1$.
\end{proposition}

\begin{proof}
  Evidently, a singleton set $A=\{i\}$ satisfies $\dd(A,P(k,n))=1$ since $\emptyset$
  is the only set covered by $A$.  Now let $A$ be a set in $P(k,n)$ with
  $|A|\ge2$.  If $|A|\le k+1$, then $A$ covers $A-\{\alpha\}$ for every $\alpha\in A$. Since the poset is graded $A$ cannot cover any further elements. Thus, $\dd(A,P(k,n))=|A|$, which by assumption is between $2$ and $k+1$.  Now suppose $|A|=k+i-1$ for some
  $i$ with $3\le i\le n-k+1$.  Then by the definition of $P(k,n)$ and the fact that it is graded $A$ covers $A-\{\alpha\}$ if and only if $\alpha\in A$ with $\alpha\ge i-1$.  We conclude that $\dd(A,P(k,n))=k+1$. 
  With this observation, the proof of the first statement of the proposition is
  complete. Furthermore, we have shown that $\dd(A,P(k,n))\leq k+1$ for all $A$ and this equality also attained. Hence, $\maxdd(P(k,n))=k+1$.
\end{proof}

With Propositions~\ref{pro:crit-pair-dd-ud} and~\ref{pro:J-critP}, we noted that
there is a $1$--$1$ correspondence between the set $J(k,n)$ of meet-irreducible 
elements of $P(k,n)$ and the set of critical pairs of $P(k,n)$.  Using the fact 
that all singletons belong to $P(k,n)$, it follows that the critical pairs
of $P(k,n)$ consist of all pairs of the form $(\{i\}, [i-1]\cup B)$ where
$[i-1]\cup B$ is a meet-irreducible element of $P(k,n)$.

The following proposition is stated for emphasis.
It is an immediate consequence of the rule for the $1$--$1$ correspondence.

\begin{proposition}\label{pro:Pkn-S2}
  Let $A=[i-1]\cup B$ and $A'=[j-1]\cup C$ be meet irreducible elements of $P(k,n)$
  with $i\le j$.  Also, let $(X,Y)$ and $(X',Y')$ be the critical pairs
  of $P(k,n)$ associated with $A$ and $A'$ respectively.  Then
  $((X,Y),(X',Y'))$ is a strict alternating cycle of size~$2$ if and only if 
  $i<j$ and $j\in B$.
\end{proposition}

The next proposition is only marginally more complex.

\begin{proposition}\label{pro:Pkn-sac}
  Let $m\ge2$, and let $(A_1,\dots,A_m)$ be a sequence of
  meet-irreducible elements  of $P(k,n)$.  For each $j\in[m]$,
  let $(X_j,Y_j)$ be the critical pair associated with
  $A_j$.   If the sequence $(((X_1,Y_1),\dots,(X_m,Y_m))$
  is a strict alternating cycle, then
  then $m=2$.
\end{proposition}

\begin{proof}
  For each $j\in[m]$, let $A_j=[i_j-1]\cup B_j$.  Then $(X_j,Y_j)=
  (\{i_j\}, A_j)$.  Since the alternating cycle is strict, the
  elements of $\{i_1,\dots,i_m\}$ are distinct. With a relabeling if necessary,
  we may assume that $i_1<i_j$ for all $j=2,\dots,m$. 

  It follows that for each $j=2,\dots,m$,  $X_1=\{i_1\}\subset 
  [i_j-1]\subset A_j=Y_j$.  Again, since the alternating cycle is strict,
  this now requires $m=2$.
\end{proof}

If $1\le k\le n-2$ and $t$ is a positive integer, a sequence
$(Y_1,\dots,Y_n)$ of subsets of $[t]$ is said to be
$(k,n)$-\emph{distinguishing} if for every meet-irreducible
element $A=[i-1]\cup B\in J(k,n)$, there is an element $\alpha\in Y_i$ 
such that $\alpha\not\in Y_j$ whenever $j\in B$.

\begin{proposition}\label{pro:dim-translation}
  If $1\le k\le n-2$, $\dim(P(k,n))$ is the least positive integer $t$
  for which there is a $(k,n)$-distinguishing sequence $(Y_1,\dots,Y_n)$ 
  of subsets of $[t]$.
\end{proposition}

\begin{proof}
  Suppose first that $\dim(P(k,n))=t$.
  We show that there is a sequence $(Y_1,\dots,Y_n)$ of subsets of $[t]$ that
  is $(k,n)$-distinguishing.  Let $S_1,\dots,S_t$ be 
  reversible sets covering all critical pairs of $P(k,n)$.  For each
  $i\in[n]$, let $Y_i$ consist of all $\alpha\in[t]$ such that
  there is a critical pair $(X,Y)=(\{i\},[i-1]\cup B)$ assigned to $S_\alpha$.
  We claim that $(Y_1,\dots,Y_n)$ is $(k,n)$-distinguishing.
  To see this, let $[i-1]\cup B$ be any meet-irreducible element
  of $P(k,n)$.  Then there is some $\alpha\in[t]$ such that the
  critical pair $(X,Y)=(\{i\},[i-1]\cup B)$ is assigned to $S_\alpha$.
  We claim that $\alpha\not\in Y_j$ for all $j$ with $j\in B$.
  If this fails, there is a meet-irreducible element $[j-1]\cup C$ such
  that the critical pair $(X',Y')=(\{j\},[j-1]\cup C)$ is also assigned to
  $S_\alpha$.  However, $i<j$ and $j\in B$ imply that $((X,Y),(X',Y'))$
  is a strict alternating cycle, contradicting the fact that $S_\alpha$
  is reversible.  We conclude that the sequence $(Y_1,\dots,Y_n)$ is 
  $(k,n)$-distinguishing as claimed.

  For the converse, suppose that $t$ is a positive integer, and there is
  a sequence $(Y_1,\dots,Y_n)$ of subsets of $[t]$ that is $(k,n)$-distinguishing.
  We show $\dim(P(k,n))\le t$.  To accomplish this, for each $\alpha\in[t]$,
  we let $S_\alpha$ consist of all critical pairs $(X,Y)=
  (\{i\},[i-1]\cup B)$ such that $\alpha\in Y_i$ and $\alpha\not\in Y_j$ whenever
  $j\in B$. Using Propositions~\ref{pro:Pkn-S2} and~\ref{pro:Pkn-sac},
  for each $\alpha\in[t]$, the set $S_\alpha$ is reversible, so $\dim(P(k,n))\le t$.
\end{proof}

For each convex geometry in the family $\{P(1,n):n\ge3\}$, we now determine the
value of the dimension \emph{exactly}. The following result is statement~5a.

\begin{proposition}
  If $n\ge3$, then $\dim(P(1,n))=1+\lfloor \lg n \rfloor$.
\end{proposition}

\begin{proof}
  Let $t=\dim(P(1,n))$ and $s=1+\lfloor\lg n\rfloor$.  Then there
  is a $(1,n)$-distinguishing sequence $(Y_1,\dots,Y_n)$ of subsets of
  $[t]$.   When $1\le i<j\le n$, the meet-irreducible
  set $[i-1]\cup \{j\}$ requires $Y_i-Y_j\neq\emptyset$.  In particular,
  this requires $Y_i\neq Y_j$ and $Y_i\neq\emptyset$.  The meet-irreducible
  set $[n-1]$ requires $Y_n\neq\emptyset$.  We have now shown that
  the sets in the sequence $(Y_1,\dots,Y_n)$ are distinct \emph{and}
  non-empty.  This requires
  $t\ge s$. 

  For the converse, consider the family of all subsets of $[s]$
  partially ordered by inclusion.  Set $m=2^s$, and consider an
  arbitrary linear extension $L$ of the Boolean algebra $\bftwo^s$.  
  Then let $(Y_1,\dots,Y_m)$ be the dual of $L$. Note that
  $Y_1=[s]$ and $Y_m=\emptyset$.  Since by assumption $n<2^s=m$, the last set $Y_n$ in
  the subsequence $(Y_1,\dots,Y_n)$ is non-empty.  It follows that this sequence
  is $(1,n)$-distinguishing. Indeed, for all $i\in [n]$ we have that $Y_i$ has an element $\alpha$ (because $Y_i\neq\emptyset$) that is in no $Y_j$ for $j>i$ (because $Y_j$ does not contain $Y_i$ because $L$ is  linear extension). Therefore, $t\le s$ by Proposition~\ref{pro:dim-translation}.
\end{proof}

The following result is statement~5b.
The proof is an elementary application of the
probabilistic method.

\begin{proposition}\label{pro:erdos}
  For $2\le k\le n-2$, $\dim(P(k,n))\le (k+1)2^{k+2}\log n$.
\end{proposition}

\begin{proof}
  Let $t=\lfloor(k+1)2^{k+2}\log n\rfloor$.  Then consider
  a sequence $\sigma=(Y_1,Y_2,\dots,Y_n)$ of random subsets of $[t]$, i.e.
  for each pair $(i,\alpha)$ with $i\in[n]$ and $\alpha\in[t]$, we assign $\alpha$
  to the set $Y_i$ with probability~$1/2$.  Assignments made for distinct pairs
  are independent.  The probability that $\sigma$ fails to be 
  $(k,n)$-distinguishing is at most: 
  \begin{equation}\label{eqn:erdos} 
    n^{k+1}(1-\frac{1}{2^{t+1}})^t.
  \end{equation}
  Simple calculations show that for the specified value of $t$, the expression
  given in~(\ref{eqn:erdos}) is less than~$1$.  It follows that a $(k,n)$-distinguishing
  sequence $\sigma=(Y_1,\dots,Y_n)$ exists.  This implies that $\dim(P(k,n))\le t$.
\end{proof}

Statement~6 follows from Theorem~\ref{thm:vcdim=se}.  However, there is
a direct proof which is more revealing of the structure of these geometries.

\begin{proposition}\label{pro:se}
  If $1\le k\le n-2$, $\se(P(k,n))=k+1$.
\end{proposition}

\begin{proof}
  For each $i\in[k+1]$, let $(X_i,Y_i)$ be the critical pair
  $(\{i\},[i-1]\cup\{i+1,\dots,i+k\})$.
  Then the elements of $\{X_1,\dots,X_{k+1}\}$ and $\{Y_1,\dots,Y_{k+1}\}$ form
  the standard example $S_{k+1}$.  This shows that $\se(P(k,n))\ge k+1$.

  Now suppose that $P(k,n)$ contains a standard example of size~$k+2$.
  Then there is a sequence $((X_1,Y_1),\dots,(X_{k+2},Y_{k+2}))$ of
  critical pairs of $P(k,n)$ such that whenever $1\le i<j\le k+2$,
  the pairs $(X_i,Y_i)$ and $(X_j,Y_j$ form a copy
  of the standard example $S_2$.  For each $j\in[m]$, let
  $X_j=\{i_j\}$ and $Y_j=[i_j-1]\cup B_j$.  Then the elements of
  $\{i_1,\dots,i_{k+2}\}$ are all distinct.  After a relabelling,
  we may assume that $i_1<i_2<\dots<i_{k+2}$.  Then
  $i_j\in B_1$ for $j=2,\dots,k+2$.  This is impossible since
  $|B_1|\le k$.  The contradiction completes the
  proof.
\end{proof}

The next two results are statements~7 and~8. In the proofs, when we
write $\{i_1,i_2,\dots,i_m\}$ is an $m$-element subset of $[n]$, we
imply $i_1<i_2<\dots<i_m$.  Readers will note that our proofs use
Ramsey theory and borrow from ideas in~\cite{BPST20}.

\begin{proposition}\label{pro:bdim-dim}
  For fixed $k\ge1$, $\bdim(P(k,n))\rightarrow\infty$ as 
  $n\rightarrow\infty$.
\end{proposition}

\begin{proof}
  In view of Proposition~\ref{pro:monotonic}, it suffices to prove the
  result when $k=1$.
  We assume that there is an integer $t\in\bfN$ such that $\bdim(P(1,n))\le t$
  and argue to a contradiction \emph{provided} $n$ is sufficiently large
  in terms of $t$.  Let $\cgR=(L_1,\dots,L_t)$ be a Boolean realizer for 
  $P(1,n)$.

  Now let $\{i_1,i_2,i_3\}$ be a $3$-element subset of $[n]$.
  Then $A=[i_1-1]\cup\{i_2\}$ and $B=[i_2-1]\cup\{i_3\}$ are meet-irreducible
  elements of $P(k,n)$.  Also, $A\parallel B$ in $P(1,n)$.
  The query string $q(A,B,\cgR)$ is a bit string
  of length~$t$, so we have defined a coloring of the $3$-element subsets 
  of $[n]$ using $2^t$ colors.  It follows that if $n$ is sufficiently large, 
  then there is a $4$-element subset $H=\{i_1,i_2,i_3,i_4\}$ of $[n]$ 
  such that each of the $3$-element subsets of $H$ is assigned the
  same color.  Let $A=[i_1-1]\cup\{i_2\}$, $B=[i_2-1]\cup\{i_3\}$ 
  and $C=[i_3-1]\cup\{i_4\}$.
  It follows that $q(A,B,\cgR)=q(B,C,\cgR)$.  
  This implies that $q(A,C,\cgR)=q(A,B,\cgR)$. This is a 
  contradiction since by $i_1<i_2<i_3<i_4$ we have $A\subseteq C$ hence $A<C$ in $P(1,n)$.
\end{proof}

\begin{proposition}\label{pro:ldim-dim}
  For fixed $k\ge1$,
  $\ldim(P(k,n))\rightarrow\infty$ as $n\rightarrow\infty$.
\end{proposition}

\begin{proof}
  In view of Proposition~\ref{pro:monotonic}, it suffices to prove the
  result when $k=1$.
  We assume that there is an integer $t\in\bfN$ such that $\ldim(P(1,n))\le t$
  and argue to a contradiction \emph{provided} $n$ is sufficiently large in terms
  of $t$.  Let $\cgR=(L_1,\dots,L_m)$ be a local realizer of $P(1,n)$ such that
  each element of $P(1,n)$ appears in at most $t$ different extensions in this
  list.

  For each $4$-element subset $\{i_1,i_2,i_3,i_4\}$ of $[n]$, we consider
  the meet-irreducible element $A=[i_1-1]\cup\{i_3\}$ and the
  meet-irreducible element $B=[i_2-1]\cup\{i_4\}$.  Then
  $A\parallel B$ in $P(1,n)$.  Let $s$ be the least integer in $[m]$ such 
  that $A>B$ in $L_s$.  Then
  there is a uniquely determined pair $(r,r')$ of integers in $[t]$ such
  that occurrence $r$ of $A$ is in $L_s$ and occurrence $r'$ of $B$ is
  in $L_s$.  

  Now we are coloring the $4$-element subsets of $[n]$ with $t^2$ colors.
  It follows that if $n$ is sufficiently large, then
  there is a $7$-element subset $H=\{i_1,\dots,i_7\}$
  such that each $4$-element subset of $H$ is assigned the
  same color, say $(r,r')$, where $r,r'\in[t]$.  

  Let $A=[i_1-1]\cup\{i_4\}$, $B=[i_2-1]\cup\{i_5\}$, $C=[i_3-1]\cup\{i_6\}$, 
  and $D=[i_5-1]\cup\{i_7\}$.  Let $s$ be the least integer in $[m]$ such that
  $A>B$ in $L_s$.  Using the set $\{i_1,i_2,i_4,i_5\}$, we know that
  occurrence $r$ of $A$ is in $L_s$, and occurrence $r'$ of $B$ is in $L_s$.
  Using the set $\{i_1,i_3,i_4,i_6\}$,  we know that occurrence $r'$ of $C$
  is also in $L_s$ with $A>C$ in $L_s$.  Using the
  set $\{i_2,i_3,i_5,i_6\}$, we know that occurrence $r$ of $B$ is in $L_s$ with
  $B>C$ in $L_s$.  This forces $r=r'$.

  Using the set $\{i_3,i_5,i_6,i_7\}$, we conclude that occurrence $r$ of $C$
  is over occurrence $r$ of $D$ in $L_s$.  This forces $A>D$ in $L_s$, which
  is a contradiction since $A<D$  in $P(1,n)$.
  The contradiction completes the proof.
\end{proof}

The final result in this section is statement~9.  

\begin{proposition}\label{pro:fdim}
  If $1\le k\le n-2$, then
  $\fdim(P(k,n))<2^{k+1}$.
\end{proposition}

\begin{proof}
  We show that $\fdim(P(k,n))<2^{k+1}$ by constructing an appropriate fractional
  realizer $f$.  

  Let $m=2^n$, and 
  let $\sigma=(Z_1,\dots,Z_{m-1})$ be a listing of the non-empty
  subsets of $[n]$.  The order of the sets in the list is arbitrary.  Then 
  for each $\alpha\in[m-1]$, there is a linear extension
  $L_\alpha$ of $P(k,n)$ that reverses all critical pairs of the form
  $(\{i\},[i-1]\cup B)$, where $i\in Z_\alpha$ and $j\not\in Z_\alpha$ whenever
  $i<j$ and $j\in B$. Namely, if $Z_\alpha=\{i_1 \ldots, i_{\ell}\}$ the ordering $[i_1-1]\cup B_{i_{1,1}},\ldots, [i_1-1]\cup B_{i_{1,r}}, \{i_1\}, [i_2-1]\cup B_{i_{2,1}}\ldots \{i_{\ell}\}$ respects the comparabilities of $P(k,n)$ since the sets $B$ are disjoint from $Z_\alpha$. Hence, $L_\alpha$ can be taken to be any linear extension respecting this ordering. 

  Then we set $f(L_\alpha)=2^{k+1}/m$ for each $\alpha\in[m-1]$.  Also, 
  we set $f(L)=0$ for all other linear extensions of $P(k,n)$.

  We now show that $f$ is a fractional realizer of $P(k,n)$.  We need only to 
  show that if $C,D\in P(k,n)$ and $C\parallel_{P(k,n)} D$, then
  the following inequality holds:
  \begin{equation}\label{eqn:easy} 
    s(C,D)=\sum\{f(L_\alpha): \alpha\in[m-1], C> D \text{ in } L_\alpha\}\ge1.
  \end{equation}
  Choose a critical pair $(X,Y)$ such that
  $X\le C$ in $P(k,n)$ and $D\le Y$ in $P(k,n)$.  If $\alpha\in[m-1]$ and
  $X>Y$ in $L_\alpha$, then $C>D$ in $L_\alpha$.  It follows that we
  need only show that inequality~(\ref{eqn:easy}) holds for every
  critical pair.

  Let $(X,Y)=(\{i\},[i-1]\cup B)$ be a critical pair.  Then the 
  number of sets in the sequence $\sigma$ that contain $i$ but
  do not contain $j$ when $i<j$ and $j\in B$ is exactly
  $2^{n-|B|-1}$, which is at least $2^{n-k-1}$.  It follows that
  \[
    s(X,Y)\ge 2^{n-k-1}\cdot \frac{2^{k+1}}{m}=1
  \]
  Further, we note that
  \[
    \sum_{i=1}^{m-1}f(L_i)=2^{k+1}\frac{m-1}{m}<2^{k+1}.
  \]
  We conclude that $\fdim(P(k,n))<2^{k+1}$.
\end{proof}

\section{Open Problems}

For posets in general, the standard examples have bounded Boolean dimension,
bounded local dimension, and unbounded dimension.  As noted
previously, the constructions given in~\cite{TroWal17} show that neither of
Boolean dimension and local dimension is bounded in terms of the other.
However, for convex geometries, we have been unable to separate 
dimension from either of Boolean dimension and local dimension.
Also, we have been unable to separate Boolean dimension and
local dimension in either direction.  We note that for the special class
of distributive lattices, the dimension coincides with the convex dimension, see~\cite{EdeJam87}. However large width of the poset of meet-irreducible elements of a distributive lattice forces a large Boolean subalgebra, which forces large Boolean dimension~\cite{BHLMM23} and large local dimension~\cite{KMMSSUW20}. Hence, separating these parameters is not
possible on distributive lattices. On the other hand, many other special classes of convex 
geometries have been identified in the literature, and we wonder whether
some of them exhibit properties that enable these questions to be 
answered.

For the family $\{P(k,n);1\le k\le n-2\}$ of convex geometries discussed
in the preceding section, we determined the value of $\dim(P(1,n))$
exactly. When $k\ge2$, we feel it is unlikely that we can obtain an exact
formula for $\dim(P(k,n))$.  However, a lower bound of the
form $c2^k\log n$ where $c$ is a positive constant should hold.

We suspect that the upper bound $\fdim(P(k,n))<2^{k+1}$ is asymptotically
tight, i.e., for every $\epsilon>0$, $\fdim(P(k,n))>2^{k+1}-\epsilon$
when $n$ is sufficiently large.

Fractional dimension is the linear programming relaxation of dimension.
Although not studied in the main body of this paper, there is a natural
linear programming relaxation of local dimension, called 
\textit{fractional local realizer}.  We refer the reader to~\cite{SmiTro21}
for the precise definition.  We wonder whether the fractional local dimension
of the convex geometry $P(k,n)$ is always strictly less than its
fractional dimension.

Convex geometries originally arise from abstracting convex subsets of finite point sets in Euclidean space and indeed every convex geometry $P$ on $[n]$ can be represented by point sets in $\mathbb{R}^d$. The minimum such $d$ is bounded by $\min(n,\cdim(P))$, see~\cite{KNO05,RG17}. Representations through hyperplane arrangements in $\mathbb{R}^d$ seem possible as well~\cite{CCK23}. What is the relation of this \emph{Euclidean dimension} to the other concepts?

Finally, we believe that the notion of VC-dimension could be extended to general posets $P$, where it can be defined as the minimum VC-dimension over all set systems whose inclusion order is isomorphic to $P$. How does this parameter related to the other concepts of dimension? The question relates to rich literature about poset embeddings into Boolean algebras, see~\cite{Trot75,BCIJ23,Wil92}. By the Boolean property, for convex geometries this notion coincides with the VC-dimension as we defined it previously. Other classes in which a natural embedding into the Boolean algebra give a VC-dimension arise from oriented matroids by Bj\"orner, Edelman, and Ziegler~\cite{BEZ90} and more generally 
from tope graphs of complexes of oriented matroids 
(see~\cite[Chapter 7.6]{Khab} for several related questions).

\subsubsection*{Acknowledgements.} KK is supported by the Spanish Research Agency
through grants RYC-2017-22701, PID2019-104844GB-I00, PID2022-137283NB-C22 and the Severo Ochoa and María de Maeztu Program for Centers and Units of Excellence in R\&D (CEX2020-001084-M) and by the French Research Agency through ANR project DAGDigDec: ANR-21-CE48-0012.

      \end{document}